\documentclass[11pt,xcolor=pdftex,a4paper,oneside,centertags,reqno]{amsart}

\usepackage{latexsym}
\usepackage{amsmath}
\usepackage{amsfonts}
\usepackage{amssymb}
\usepackage{graphpap}
\usepackage{epsfig}
\usepackage{amscd}
\usepackage{amsthm}
\usepackage{enumerate}
\usepackage{color}
\usepackage{eucal} 
\usepackage[francais]{[babel}
\usepackage[ansinew]{inputenc}

\font\got=eufm10 at 12pt

\input cyracc.def

\newcommand{\nor}[1]{|\hskip -0.6pt | #1 |\hskip -0.6pt |}

\newcommand{\mn}[2]{\{ #1\, ;\, #2 \}}
\newcommand{\sk}[2]{\left\langle #1 , #2\right\rangle}
\newcommand{\skal}[2]{\left( #1 \cdot #2\right)}

\newcommand{\divg}[0]{{\rm div}}

\renewcommand{\geq}[0]{\geqslant}
\renewcommand{\leq}[0]{\leqslant}
\renewcommand{\i}[0]{{\rm i}}

\newcommand{\C}[0]{{\mathbb C}}

\newcommand{\R}[0]{\mathbb{R}}

\newcommand{\pd}[0]{\partial}

\newtheorem{thm}{Theorem}
\newtheorem{lema}{Lemma}
\newtheorem{prop}{Proposition}

\begin{document}

\title[Operators in divergence form with potentials]
{Bilinear embedding for real elliptic differential operators in divergence form with potentials}

\author[Dragicevic]{Oliver Dragi\v{c}evi\'c}
\author[Volberg]{Alexander Volberg}


\address{Oliver Dragi\v{c}evi\'c \\Department of Mathematics\\ Faculty of Mathematics and Physics\\ University of Ljubljana\\ Jadranska 19, SI-1000 Ljubljana\\ Slovenia}
\email{oliver.dragicevic@fmf.uni-lj.si}

\address{Alexander Volberg\\Department of Mathematics\\ Michigan State University\\ East Lansing, MI 48824\\ USA}
\email{volberg@math.msu.edu}

\thanks{The first author was partially supported by the Ministry of Higher Education,
Science and Technology of Slovenia (research program Analysis and Geometry, contract no. P1-0291). The second author was partially supported by NSF grants DMS-0200713, 0501067, 0758552}

\begin{abstract}
We present a simple Bellman function proof of a bilinear estimate for elliptic operators in divergence form with real coefficients  and with nonnegative potentials. The constants are dimension-free. The $p$-range of applicability of this 
estimate is $(1,\infty)$ for any real accretive (nonsymmetric) matrix $A$ of coefficients. 
\end{abstract}

\maketitle

\section{Motivation}
\label{motivation}

The bilinear-embedding-type theorems as the one we prove below for general second order operators in divergence form with potentials, $Lu := -\divg (A \nabla u)+Vu$, where $A$ is a real accretive matrix and $V$ a nonnegative locally integrable function on $\R^n$, have so far played an important part in estimates of several singular integral operators. 
In \cite{PV,DV2} the bilinear embedding was used for $L^p$ estimates of the Ahlfors-Beurling operator $T$. 
In particular, that implied a regularity result for solutions of Beltrami equations at the critical exponent, which, in
its turn, had a geometric significance for the theory of quasiregular maps; see \cite{AIS}. 

When applied to families of operators on spaces of arbitrary dimensions, the bilinear embedding tends to give dimension-free estimates with sharp behaviour in terms of $p$.
The first such example is \cite{DV},
where dimension-free estimates were obtained for Riesz
transforms associated to the classical and the Ornstein-Uhlenbeck Laplacians. 
This approach was continued in
\cite{DV-Sch} 
in the context of arbitrary Schr\"odinger operators with nonnegative potentials. As a special case, dimension-free estimates for the Hermite-Riesz transforms on $L^p$ were confirmed, with linear behaviour with respect to $p$.

Recently, A. Carbonaro and the first author \cite{CD} studied a class of Riesz transforms in the fairly general setting of  complete Riemannian manifolds whose Bakry-Emery Ricci curvature is bounded from below. The study of such operators was initiated by Strichartz \cite{S}. Bakry \cite{B} proved the first dimension-free estimates, whereupon the work of Li \cite{Li} implied that in the general case the behaviour of the norms with respect to $p$ is at most quadratic as $p\rightarrow\infty$. In \cite{CD} this was improved by giving linear estimates. They are sharp. Also, while the proofs in \cite{B} and \cite{Li} are probabilistic, \cite{CD} features the first analytic proof. The bilinear embedding was instrumental for obtaining all of these results.

The operators we consider are also called {\it generalized Schr\"odinger operators}, see \cite
{D}. When $A\equiv I$ we have the usual Schr\"odinger operators $-\Delta+V$. 
The fact that $A\equiv I$ makes them comparatively easy in the
theory of general second order accretive operators in divergence
form, so the emphasis in, say, \cite{DV, DV-Sch, CD, DV1} is on
independence of estimates on the dimension of the underlying manifold.

Yet another place where the bilinear embedding
featured
(under a slight 
disguise) 
is
the sharp estimate of the weighted Hilbert transform in
terms of the $A_p$ norm of the weight \cite{P,PW}.

Of course the dimension-free estimates of Riesz transforms were
known before. One can find an extensive bibliography in \cite{DV} and \cite{Li}.
We think the approach in \cite{DV, DV-Sch, CD, DV1} is a sort of a
unified one. 
Namely, there 
the 
proofs are 
divided into two steps: 1) the
proof of the bilinear embedding (always dimension-free) and 2) a formula
which involves holomorphic calculus of the operator $L$.
We would like to emphasize the importance of the second step
because of its relations with Kato's problem. The holomorphic
calculus of general second order accretive operators in
divergence form also plays crucial part in proving estimates
for the related Riesz transforms, see e.g.
\cite{A, AMN}.

Finally, we would like to acknowledge that the completion of this paper was decisively motivated by recent appearance of a closely related  preprint \cite{AHM} by Auscher, Hoffman and Martell. 
In particular, \cite{AHM} 
explains the connection between the results in this paper and square function estimates in the context of accretive matrices.

\bigskip
\noindent{\bf Acknowledgements.} We are grateful to P. Auscher, S.
Hofmann, C. Kenig and J. Pipher who advised us on general second
order elliptic operators. Special thanks are to J. M. Martell for
several valuable conversations and remarks.

We are particularly grateful to Pascal Auscher whose interest made us resume working on this paper after a pause of over three years. 

\section{Elliptic differential operators in divergent form}
\label{elliptic}

Let $A=[a_{ij}(x)]_{i,j=1}^n$ be a real accretive $n\times n$
matrix function on $\R^n$ with $L^\infty$ coefficients which is uniformly elliptic, meaning that for some $\gamma>0$ and $\forall
\xi\in \C^n, \, x\in\R^n$,
\begin{equation}
\label{elli} 
\text{Re\,}(A(x)\xi\cdot\overline\xi) \geq \gamma \|\xi\|^2\,,
\end{equation}
where $z\cdot w:=\sum_1^nz_jw_j$ for $z,w\in\C^n$.

Let $V$ be a nonnegative locally integrable function on $\R^n$. We consider the operator
$$
Lu := -\,\divg (A \nabla u) +Vu= -\sum_{i,j=1}^n\frac{\pd}{\pd x_i}
\Big(a_{ij}(x) \frac{\pd u}{\pd x_j}(x)\Big ) +V(x)u(x) \,
$$
and the associated semigroup $(P_t :=e^{-tL})_{t\geqslant 0}$. For precise definitions see \cite[Section 1.8]{D} or \cite[Section 4.7]{O}.
Denote furthermore $\tilde f(x,t)=(P_tf)(x)$.

Next, for a smooth complex-valued function $\phi=\phi(x,t)$ on $\R^n\times(0,\infty)$ introduce 
$|\phi|_*^2:=|\nabla\phi|^2+V|\phi|^2$, where
$x=(x_1,\hdots,x_n)$ and 
$\nabla=(\pd_{x_1},\hdots,\pd_{x_n})$ is the spatial gradient. That is,
\begin{equation}
\label{vaha}
|\phi(x,t)|_*^2:=\sum_{j=1}^n\Big|\frac{\pd \phi}{\pd x_j}(x,t)\Big|^2 + V(x)|\phi(x,t)|^2\,.
\end{equation}
Take any number $p\in [2,\infty)$ and denote by $q$ its conjugate exponent, i.e. $1/p+1/q=1$. What follows is the main result of the paper, which we call {\it bilinear embedding theorem}.
\begin{thm}
\label{bilinreal} 
Let the operator $L$ be as above. For any $f,g\in C_c^\infty(\R^n)$ we have
$$
\int_0^\infty\int_{\R^n} |\tilde f(x,t)|_*|\tilde g(x,t)|_*\,dx\,dt\leq 
C_\gamma p\,\nor{f}_p\nor{g}_q\,.
$$
\end{thm}
The theorem is also valid for $1<p<2$, in which case the factor on the right becomes $C_\gamma q$. Moreover, it will emerge from the proof of the theorem that one might take $C_\gamma=C\max\{1,\gamma^{-1}\}$, where $C$ is some absolute constant.

When $V=0$, Theorem \ref{bilinreal} 
would obviously follow if we had the following estimate from above:
\begin{equation}
\label{S1} \int_{\R^n} (G_Lu(x))^p dx \leq C\int_{\R^n} |u(x)|^p\, dx
\end{equation}
for all $p\in (1,\infty)$, 
where $G_Lu$ is the square function, i.e.
$$
G_Lu(x) := \Bigl(\int_{0} ^{\infty}|\nabla P_tu(x)|^2\,dt\Bigr)^{\frac12}\,.
$$
In general this inequality is  proved for 
$p\in(q_-(L), q_+(L))$.
Even when $A$ is a real matrix, \eqref{S1} holds only  for $p \in(1, 2+\varepsilon(n))$. See \cite[Corollaries 6.3 - 6.7]{A}.
On the other hand, an interesting paper by Auscher, Hofmann and Martell \cite{AHM} shows that the {\it cone} square function estimate for {\it real} matrices $A$ can be made for the whole range of $p\in (1,\infty)$. The proof seems to be rather delicate, and as far as it stands now the estimates depend on the dimension. Our theorem has a strange feature of being dimension-free.

\section{Bellman function}
\label{Bellman}

The Bellman function technique was introduced into harmonic analysis by Nazarov, Treil and the second author in the 1994 preprint version of their paper \cite{NTV}. The main tool in our proofs will be a particular Bellman function, drawn from a function defined by Nazarov and Treil \cite{NT}. Their example was later extended by the present authors in \cite{DV-Sch}. In this paper we work with a simplified variant of the function from \cite{DV-Sch}; it comprises only two variables, while the functions from \cite{NT, DV-Sch} have four.

Throughout this section we assume that $p\geqslant 2$, $q=p/(p-1)$ and $\delta=q(q-1)/8$ are fixed. Observe that $\delta\sim (p-1)^{-1}$.

Let $\phi:\R_+\times\R_+\longrightarrow\R_+$ be defined by
\begin{equation}
\label{foreman}
\phi(u,v)=
u^{p}+v^{q}+\delta
\left\{
\aligned
& u^2v^{2-q} & ; & \ \ u^p\leqslant v^q\\
& \frac{2}{p}\,u^{p}+\left(\frac{2}{q}-1\right)v^{q}
& ; &\ \ u^p\geqslant v^q\,.
\endaligned\right.
\end{equation}

The Bellman function we use is simply the function $Q:\C\times\C\longrightarrow\R_-$, defined by 
\begin{equation}
\label{rupkina}
Q(\zeta,\eta):=-\frac12 \phi(|\zeta|,|\eta|)\,.
\end{equation}

Our proofs will to a great extent involve estimates of the first- and second-order derivatives of $Q$. Estimating the gradient of $\phi$ is straightforward:
\begin{equation}
\label{mravalzhamieri}
\pd_u\phi\leqslant C(p)\max\{u^{p-1},v\}
\hskip 30pt 
\text{and}
\hskip 30pt 
\pd_v\phi\leqslant Cv^{q-1}\leqslant C(v+1)\,.
\end{equation}
For the last inequality we used that $q-1\leqslant 1$. Now the chain rule naturally implies the estimates of $\pd_\zeta Q$ and $\pd_\eta Q$.

Yet while $Q$ is of class $C^1$, it is not globally $C^2$, because $\phi$ fails to be $C^2$ along the curves $\{v=0\}$ and $\{u^p=v^q\}$ in $\R_+^2$. This proved not to be a major obstacle, since smoothening $Q$ by standard mollifiers enables one in various contexts to apply $Q$ completely rigorously, see \cite{DV-Sch, CD, DV}. 
Hence, for the sake of the clarity of the presentation, we will further proceed as if $Q$ were $C^2$ everywhere, referring the reader to \cite{DV-Sch, CD, DV} for the mollifying procedure. A similar smoothening argument was also carried out in \cite{PV}.
With this understanding, the properties of $Q$ are summarized in the theorem below.

\begin{thm}
\label{bejaz}
For any $u,v\geqslant 0$,
\begin{enumerate}[{\rm (i)}]
\item
\label{anterija}
$0\leqslant \phi(u,v)\leqslant (1+\delta)(u^p+v^q)$.
\end{enumerate}
If $\xi=(\zeta,\eta)\in \C\times\C$ 
then there exists  $\tau=\tau(|\zeta|,|\eta|)>0$ such that \begin{enumerate}[{\rm (i)}]
\addtocounter{enumi}{1}
\item
\label{kupres}
$
{-d^2 Q(\xi) \geqslant \delta\big(\tau|d\zeta|^2+\tau^{-1}|d\eta|^2\big)}$
\item
\label{blagaj}
$
Q(\xi)-
dQ(\xi)\,\xi\geqslant \delta\big(\tau |\zeta|^2+\tau^{-1}|\eta|^2\big).
$
\end{enumerate}
\end{thm}

The first property follows immediately from the definition of $\phi$, while the last two were proved in \cite{DV-Sch}.

Let us clarify the notation appearing in the preceding theorem.
Denote by $z_j$ the variables of $Q$, i.e. 
$$
(z_1,z_2):=(\zeta,\eta)\,.
$$

Write
$z_j=\text{\got x}_j+\i \text{\got y}_j$ for $j=1,2$, so that 
$$
\frac{\pd}{\pd z_j}=\frac12\Big(\frac{\pd}{\pd\hskip 0.7pt \text{\got x}_j}-\i\frac{\pd}{\pd \text{\got y}_j}\Big) 
\hskip 20pt \text{and} \hskip 20pt 
\frac{\pd}{\pd\bar z_j }=\frac12\Big(\frac{\pd}{\pd\hskip 0.7pt \text{\got x}_j}+\i\frac{\pd}{\pd \text{\got y}_j}\Big)\,. 
$$
Take $\sigma=(\sigma_1,\sigma_2)\in\C^2$. Then we define
\begin{equation}
\label{pazardzhiskho horo}
dQ(\xi)\sigma=\sum_{j=1}^2\Big(\frac{\pd Q}{\pd z_j}(\xi)\,\sigma_j+\frac{\pd Q}{\pd \bar z_j}(\xi)\,\bar\sigma_j\Big)\,.
\end{equation}
As for the second differential form, when $\xi, \sigma,\varsigma\in\C^2$, define 
\begin{equation*}
\label{pogorelic scherzo}
\aligned
\sk{d^2Q(\xi)\sigma}{\varsigma}:=\hskip -50pt &  \\
&\sum_{k,m=1}^2
\left(
\frac{\pd^2 Q}{\pd z_k\pd z_m}(\xi)\sigma_k\varsigma_m+
\frac{\pd^2 Q}{\pd\bar z_k\pd z_m}(\xi)\bar \sigma_k\varsigma_m+
\frac{\pd^2 Q}{\pd z_k\pd\bar  z_m}(\xi)\sigma_k\bar \varsigma_m+
\frac{\pd^2 Q}{\pd\bar  z_k\pd\bar  z_m}(\xi)\bar \sigma_k\bar \varsigma_m
\right)\,.
\endaligned
\end{equation*}
Note that this expression is always real. It is also symmetric, in the sense that 
\begin{equation}
\label{argerich}
\sk{d^2Q(\xi)\sigma}{\varsigma}=\sk{d^2Q(\xi)\varsigma}{\sigma}\,.
\end{equation}

The meaning of \eqref{kupres} above is the inequality
$$
\sk{-d^2Q(\xi)\sigma}{\sigma}\geqslant \delta\big(\tau |\sigma_1|^2+\tau^{-1}|\sigma_2|^2\big)\,,
$$
while \eqref{blagaj} should be understood through \eqref{pazardzhiskho horo}, with $\sigma=\xi$, of course.

\subsection{Setup of the proof}
Given $C_c^\infty$ functions $f,g:\R^n\rightarrow\C$, define
\begin{equation}
\label{ella sunshine}
v_1(x,t)=P_t f(x)
\hskip 30pt \text{and}\hskip 30pt
v_2(x,t)=P_t g(x)\,
\end{equation}
and $v:=(v_1,v_2)$, i.e.
\begin{equation*}
\label{nina spirova}
\aligned
v(x,t) = (P_t f(x), P_t g(x))\,.
\endaligned
\end{equation*}
Put $b:=Q\circ v$, that is,
\begin{equation}
\label{zelen bore}
b(x,t) = Q(P_t f(x), P_t g(x))\,.
\end{equation}
When proving Theorem \ref{bilinreal} we will try to imitate \cite{DV}. Let us explain what this means. 

Suppose
$\psi\in C^\infty_c(\R^n)$ is a radial function, $\psi\equiv 1$ in
the unit ball, $\psi\equiv 0$ outside the ball of radius $2$, and
$0< \psi< 1$ everywhere else. For $R>0$ define 
$$
\psi_R(x) := \psi\Big(\frac{x}R\Big)\,.
$$
Furthermore, let
\begin{equation*}
\label{kafu}
L' := \frac{\pd}{\pd t}+ L\,.
\end{equation*}
Take $R,T>0$ and consider the integral
$$
I_{R,T}(f,g):=\int_0^T\int_{\R^n}
\psi_R(x)\,L'b(x,t) \,dx\,dt\,,
$$
where $f,g$ are the same functions as in \eqref{zelen bore}.
The proof of Theorem \ref{bilinreal} will rest on the following two estimates of $I_{R,T}$ (lower and upper, respectively). 

\begin{prop}
\label{lovcen}
Let $\gamma,\delta$ be as in \eqref{elli} and \eqref{foreman}. 
For every $(x,t)\in\R^n\times (0,\infty)$ we have
\begin{equation*}
\label{eqbelow} 
L' b(x,t)\geq
2\delta \min\{1,\gamma\}|\tilde f(x,t)|_*|\tilde g(x,t)|_*\,.
 \end{equation*}
\end{prop}

\begin{prop}
\label{planinom}
For arbitrary test functions $f, g$ and $T>0$,
\begin{equation*}
\label{eqabove} 
\limsup_{R\rightarrow\infty}\int_0^T\int_{\R^n}
\psi_R(x) L' b(x,t) \, dx\, dt\leq \|f\|_p^p +
\|g\|_q^q\,.
\end{equation*}
\end{prop}

\begin{proof}[Proof of Theorem \ref{bilinreal}.]
We use the standard trick of ``polarization".
Together the Fatou lemma 
and Propositions \ref{lovcen} and \ref{planinom} imply
$$
\int_0^\infty\int_{\R^n} |\tilde f(x,t)|_*|\tilde g(x,t)|_*\,dx\, dt\leq
\frac{\max\{1,\gamma^{-1}\}}{2\delta}\,(\|f\|_p^p + \|g\|_q^q)\,.
$$
Now replace $f,g$ by $\lambda f$, $\lambda^{-1}g$, respectively, and take the infimum over $\lambda>0$.
\end{proof}

So what remains is to prove Propositions \ref{lovcen} and \ref{planinom}. This will be done in Sections \ref{below} and \ref{above}, respectively.

\section{Proof of Proposition \ref{lovcen}}
\label{below}

We would like to apply properties \eqref{kupres} and \eqref{blagaj} of the function $Q$ that were given in Theorem \ref{bejaz}. To this end we must first link them to $L'b$.

\begin{lema}
\label{plav}
Fix $(x,t)\in\R^{n+1}_+$. 
If $[a_{ij}]$ is a real accretive $n\times n$ matrix function, then
\begin{equation}
\label{chainrule}
\aligned
L' b (x,t)=\sum_{i,j=1}^n a_{ij}(x) & \sk{-d^2 Q( v) 
\frac{\pd  v}{\pd x_i}(x,t)}{\frac{\pd  v}{\pd x_j}(x,t)} 
+V(x)
[Q( v)-dQ( v) v]
 \,.
\endaligned
\end{equation}
\end{lema}
\begin{proof}
By applying the chain rule repeatedly we calculate
$$
\frac{\pd b }{\pd x_j}(x,t)=
\sum_{k=1}^2\Big[
\frac{\pd Q}{\pd z_k}( v(x,t))\frac{\pd v_k}{\pd x_j}(x,t)+
\frac{\pd Q}{\pd
\overline{z_k}}( v(x,t))\frac{\pd \overline{v_k}}{\pd x_j}(x,t)\Big]\,
$$
and
$$
\frac{\pd^2 b }{\pd x_i\pd x_j} = 
\sk{d^2 Q(v) 
\frac{\pd  v}{\pd x_i}}{\frac{\pd  v}{\pd x_j}} 
+ \sum_{k=1}^2
\left(
 \frac{\pd Q}{\pd z_k}(v)\frac{\pd^2 v_k}{\pd x_i\pd x_j}
+\frac{\pd Q}{\pd \overline{z_k}}(v)\frac{\pd^2
\overline{v_k}}{\pd x_i\pd x_j}\right)\,.
$$

Now we are ready to compute
\begin{equation}
\label{13/4} 
L'b ={\pd b \over \pd t}- \divg (A\nabla b )+Vb \,.
\end{equation}
Fix $k$ and consider only those terms in \eqref{13/4} which
contain 
${\pd_{z_k} Q}( v)$
as a factor. We get
\begin{equation*}
{\pd Q\over\pd z_k}( v(x,t))\,\Big[\frac{\pd v_k}{\pd t}(x,t)-\divg(A\nabla v_k)(x,t)\Big]\,,
\end{equation*}
which is the same as 
\begin{equation*}
{\pd Q\over\pd z_k}( v(x,t))\,(L'v_k)(x,t)-{\pd Q\over\pd z_k}(v(x,t))\,V(x)v_k(x,t)\,.
\end{equation*}
But each $v_k$ is, by \eqref{ella sunshine}, a function of the form $e^{-Lt}\varphi$, therefore
$L'v_k=0$.

If we now factor out 
${\pd_{\bar z_k}Q}( v)$,
we get
\begin{equation*}
\label{puklaura0}
 {\pd Q\over\pd \overline{z_k}}
\Big(\overline{\frac{\pd v_k}{\pd t}-\divg (\overline{A}\,\nabla
v_k)+V(x)v_k}\Big)- {\pd Q\over\pd \overline{z_k}}V(x)\overline{v_k}\,,
\end{equation*}
where $\overline{A}$ is the matrix with entries
$\{\overline{a_{ij}}\}$. 
Since $A$ was assumed to be real, the term in parentheses is again equal to $L'v_k$, hence it disappears.

Putting together the remaining terms gives \eqref{chainrule}.
\end{proof}

\begin{proof}[Proof of Proposition \ref{lovcen}.]

We would like to rewrite 
\begin{equation}
\label{seminar0211}
\sum_{i,j=1}^n a_{ij}(x) \sk{d^2 Q( v) 
\frac{\pd  v}{\pd x_i}(x,t)}{\frac{\pd v}{\pd x_j}(x,t)} \,
\end{equation}
in terms of quadratic forms, i.e. in a way where instead  of $\pd_{x_i} v$ and $\pd_{x_j} v$ we would get identical vectors. This would clear the road for applying Theorem \ref{bejaz} \eqref{kupres}. 

For any square matrix  $A=[a_{ij}]$ let 
\begin{equation}
\label{mozartkugel}
\tilde a_{ij}={a_{ij}+a_{ji}\over 2}\,
\end{equation}
and 
$
\widetilde A=[\tilde a_{ij}].
$
If $A$ is real accretive, then so is $\widetilde A$, but it is, in addition, also symmetric.

First note that the symmetry \eqref{argerich} quickly implies 
$$
\eqref{seminar0211}=
\sum_{i,j=1}^n \tilde a_{ij}(x) \sk{d^2 Q( v) 
\frac{\pd  v}{\pd x_i}(x,t)}{\frac{\pd v}{\pd x_j}(x,t)} \,.
$$
Consequently, we get
\begin{equation*}
\label{danrusije} 
\aligned 
\eqref{seminar0211}&=\sum_{k,m=1}^2 \frac{\pd^2 Q}{\pd z_k \pd z_m}(v)\sum_{i,j=1}^n \tilde a_{ij}(x)\frac{\pd v_k}{\pd
x_i}\frac{\pd  v_m}{\pd x_j} 
+\sum_{k,m=1}^2 \frac{\pd^2 Q}{\pd z_k  \pd\bar z_m}(v)\sum_{i,j=1}^n\tilde a_{ij}(x)
\frac{\pd v_k}{\pd x_i}\  {\frac{\pd\overline{v_m}}{\pd x_j}}
\\ & 
+\sum_{k,m=1} ^2 
\frac{\pd^2 Q}{\pd\bar z_k \pd z_m}(v)\sum_{i,j=1}^n\tilde a_{ij}(x)
\frac{\pd\overline{v_k}}{\pd x_i}\frac{\pd v_m}{\pd x_j}
+\sum_{k,m=1}^2  \frac{\pd^2 Q}{\pd\bar z_k \pd\bar z_m}(v) \sum_{i,j=1}^n\tilde a_{ij}(x)\frac{\pd\overline{v_k}}{\pd x_i}\,
\frac{\pd\overline{v_m}}{\pd x_j} \,.
\endaligned
\end{equation*}
This is the same as
$$
\aligned 
\sum_{k,m=1}^2 & \frac{\pd^2 Q}{\pd z_k \pd z_m}(v) \skal{\widetilde A\nabla v_k}{\nabla v_m}
+\sum_{k,m=1}^2 \frac{\pd^2 Q}{\pd z_k \pd\overline z_m}(v) \skal{\widetilde A \nabla v_k}{\nabla \overline{v_m}}
\\ \ 
+\sum_{k,m=1}^2 & \frac{\pd^2 Q}{\pd\overline z_k \pd z_m}(v) \skal{\widetilde A\nabla \overline{v_k}}{\nabla v_m}
+\sum_{k,m=1}^2 \frac{\pd^2 Q}{\pd \overline z_k \pd\overline z_m}(v) \skal{\widetilde A\nabla \overline{v_k}}{\nabla \overline{v_m}}
\,.
\endaligned
$$

Matrix $\widetilde A$ is real and symmetric. The accretivity condition implies that it is also positive-definite, hence it admits a real (positive-definite) square root $\widetilde A^{1/2}$. By introducing new vectors 
$$
\Theta_k: = \widetilde A^{1/2} \nabla v_k=:(\theta^k_1,\hdots,\theta^k_n)
$$ 
and using that $\widetilde A^{1/2}\nabla \overline{v_k}=\overline{\widetilde A^{1/2} \nabla v_k}$ we can rewrite the last sum again as
$$
\aligned
\sum_{k,m=1}^2 
& \left( 
    \frac{\pd^2 Q}{\pd     z_k \pd      z_m} (v) \skal{\Theta_k}{\Theta_m}
  + \frac{\pd^2 Q}{\pd     z_k \pd \bar z_m} (v) \skal{\Theta_k}{\overline{\Theta_m}} + \frac{\pd^2 Q}{\pd\bar z_k \pd z_m} (v) \skal{\overline{\Theta_k}}{\Theta_m}
  \right.\\
&
\left.+ \frac{\pd^2 Q}{\pd\bar z_k \pd \bar z_m} (v) \skal{\overline{\Theta_k}}{\overline{\Theta_m}}
\right)\\
 = 
\sum_{j=1}^n & \sum_{k,m=1}^2 
\left( 
\frac{\pd^2 Q}{\pd z_k \pd z_m} (v)\,\theta^k_j \theta^m_j 
+\frac{\pd^2 Q}{\pd z_k \pd\bar z_m} (v)\,\theta^k_j \bar\theta^m_j
+\frac{\pd^2 Q}{\pd\bar z_k \pd z_m} (v)\,\bar\theta^k_j \theta^m_j 
+\frac{\pd^2 Q}{\pd\bar z_k \pd\bar z_m} (v)\,\bar\theta^k_j  \bar\theta^m_j
\right)
\,.
\endaligned
$$
This expression we can write in turn as 
$\sum_{j=1}^n \sk{d^2Q(v)\theta_j}{\theta_j}$, 
where 
$\theta_j:=(\theta_j^1,\theta_j^2)\in\C^2$ for $j\in\{1,\hdots,n\}$.
So we proved 
$$
\sum_{i,j=1}^n a_{ij}(x) \sk{d^2 Q( v) 
\frac{\pd  v}{\pd x_i}(x,t)}{\frac{\pd v}{\pd x_j}(x,t)} 
=\sum_{j=1}^n \sk{d^2Q(v)\theta_j}{\theta_j}\,.
$$
We can finally apply Lemma \ref{plav} and properties \eqref{kupres}, \eqref{blagaj} from Theorem \ref{bejaz}. The result is 
$$
\aligned
L'b(x,t) &= \sum_{j=1}^n
\sk{-d^2Q(v)\theta_j}{\theta_j}+V(x)[Q( v)-dQ( v)\, v]\\
&\geq \delta\sum_{j=1}^n
\left(\tau|\theta_j^1|^2+\tau^{-1}|\theta_j^2|^2\right)+V(x)\delta\!\left(\tau |v_1|^2+\tau^{-1}|v_2|^2\right)\\
& \geq
2\delta\sqrt{V|v_1|^2+\sum_{j=1}^n|\theta_j^1|^2}\sqrt{V|v_2|^2+\sum_{j=1}^n|\theta_j^2|^2}\\
&=2\delta\sqrt{V|v_1|^2+|\Theta_1|^2}\sqrt{V|v_2|^2+|\Theta_2|^2}
\,.
\endaligned
$$
Recall that $\Theta_k=\widetilde A^{1/2}\nabla v_k$ and 
use that 
$\widetilde A\xi\cdot\overline\xi=\text{Re\,}(A\xi\cdot\overline\xi)$. 
Therefore, \eqref{elli} gives
$$
\geq 2\delta\min\{1,\gamma\}\sqrt{V|v_1|^2+|\nabla v_1|^2}\sqrt{V|v_2|^2+|\nabla v_2|^2}
\,.
$$
Finally, by \eqref{vaha} and \eqref{ella sunshine},
this is the same as
$
2\delta\min\{1,\gamma\}|\tilde f(x,t)|_*|\tilde g(x,t)|_*\,.
$
\end{proof}

\section{Integration by parts: proof of Proposition \ref{planinom}}
\label{above}

In order to prove Proposition \ref{planinom} we will 
first integrate by parts, as summarized in the following inequality.

\begin{lema}
\label{byparts} 
Let $A$ be a real accretive matrix (not necessarily symmetric) and  
 $p\geq 2$.
Then
\begin{equation*}
\label{byparts1} 
\limsup_{R\rightarrow\infty} \int_0^T\int_{\R^n}
\psi_R(x)L' b (x,t)\, dx\, dt
\leqslant 
\limsup_{R\rightarrow \infty}
\int_0^T\int_{\R^n} \psi_R(x)\,\frac{\pd b }{\pd t}(x,t) \,dx\,dt\,.
\end{equation*}
\end{lema}
\begin{proof}
We first recall \eqref{13/4} and notice that $Vb \leqslant 0$. Hence it suffices to prove that
$$
\limsup_{R\rightarrow\infty} \int_0^T\int_{\R^n}
\psi_R(x)\,\divg (A\nabla b )(x,t)\, dx\, dt=0\,.
$$
Because of the integration by parts and the uniform boundedness of $A$ it is sufficient to show 
\begin{equation*}
\label{byparts2} 
\lim_{R\rightarrow \infty}
\int_0^T\int_{\omega_R}
|\nabla \psi_R (x)|\,|\nabla b (x,t)|
\,dx \, dt=0\,.
\end{equation*}
Here $\omega_R:={\rm supp}\,\nabla\psi_R=\mn{x\in\R^n}{R\leqslant|x|\leqslant 2R}$. When 
$R$ is large, 
\begin{equation}
\label{scarlatti k141}
d(\omega_R,\text{supp\,}f)\sim R\,. 
\end{equation}

Since 
$|(\nabla \psi_R )(x)|\leqslant R^{-1}\nor{\nabla\psi}_\infty$, our inequality boils down to showing
\begin{equation*}
\label{byparts2} 
\lim_{R\rightarrow \infty}
\int_0^T\int_{\omega_R}
|\nabla b (x,t)|\,dx \, dt
=0\,.
\end{equation*}
By \eqref{zelen bore}, this will in turn follow once we demonstrate
\begin{equation}
\label{gledala}
\lim_{R\rightarrow \infty}
\int_0^T\int_{\omega_R}
\left(
\Big|\frac{\pd Q}{\pd \zeta}( v)\Big| |\nabla P_t f(x)|+
\Big|\frac{\pd Q}{\pd \eta} ( v)\Big| |\nabla P_t g(x)|
\right)
\,dx \, dt=0\,.
\end{equation}
For that purpose we resort to the so-called {\it off-diagonal
estimates} for operators $L$ which were proven in \cite[Section 2.3]{A}. 
That is
a result stating that for any $h$, supported in a closed set
$E\subset \R^n$, the following holds:
\begin{equation}
\label{offdiag1} \|T_t h\|_{L^2(F)} \leq C
e^{-\frac{c}{t}\,d(E,F)^2} \|h\|_{L^2(E)}\,.
\end{equation}
Here $F\subset \R^n$ is closed and $ T_t $ can be anything from
$P_t,  tLP_t, \sqrt{t}\nabla P_t$.
While the proof in \cite{A} was done for the case $V=0$, it can be repeated so as to cover the case of general nonnegative potential $V\in L_{loc}^1(\R^n)$.

In order to estimate the partial derivatives of $Q$ with respect to $\zeta$ and $\eta$ we first recall \eqref{mravalzhamieri}, which suggests considering separately the domains
$$
\Lambda_1 := \mn{(\zeta,\eta)\in\C^2}{|\zeta|^p \leq |\eta|^q}
\,\hskip 30pt {\rm and}\,\hskip 30pt 
\Lambda_2 := \mn{(\zeta,\eta)\in\C^2}{|\zeta|^p > |\eta|^q}
\,.
$$
The upper estimates of $\nabla Q$ in $\Lambda_1$ and $\Lambda_2$ can now be directly inferred from  \eqref{rupkina} and \eqref{mravalzhamieri}.

Fix $t\geqslant 0$ and define, for  $j\in\{1,2\}$,
$$
{\mathcal V}_j={\mathcal V}_j(t):=\mn{x\in\R^n}{ v(x,t)\in\Lambda_j}\,.
$$

\begin{enumerate}[1.)]
\item
Let us estimate $|\pd_{\zeta} Q(v)| |\nabla P_t f|$ in domain ${\mathcal V}_1$, that is, when $|P_t f|^p \leq |P_t g|^q$. This condition is equivalent to $|P_t f|^{p-1} \leq |P_t g|$, therefore, by  \eqref{rupkina} and \eqref{mravalzhamieri},
\begin{align}
\label{lile}
 \Big|\frac{\pd Q}{\pd \zeta}( v)\Big|
\leq C \,|P_t g| \,.
\end{align}
This and the H\"older's inequality imply
$$
\int_{\omega_R} 
{\chi}_{{\mathcal V}_1}\,
\Big|\frac{\pd Q}{\pd \zeta}(v)\Big|
|\nabla P_t f|
\leq 
\frac{C}{\sqrt t}\, 
\nor{P_tg}_{L^2(\omega_R)}\nor{\sqrt t\nabla P_tf}_{L^2(\omega_R)}
\,.
$$
By the off-diagonal estimates \eqref{offdiag1} and the similarity \eqref{scarlatti k141} we get
$$
\leq \frac{C(f,g,n)}{\sqrt t}\,e^{-\frac ct \, R^2}\,.
$$
Finally, integrate in $t$ and send $R$ to infinity. Now it is clear that
$$
\lim_{R\rightarrow\infty}\int_0^T\int_{\omega_R} 
{\chi}_{{\mathcal V}_1(t)}(x)\,
\Big|\frac{\pd Q}{\pd \zeta}(v(x,t))\Big|
|\nabla P_t f(x)|
\,dx \,dt=0\,.
$$

\item
In the domain ${\mathcal V}_2$ we similarly get
\begin{equation*}
\label{11/4} 
\Big|\frac{\pd Q}{\pd \zeta}(v)\Big|
\leq  
C |P_t f|^{p-1} \,.
\end{equation*}
Applying the H\"older's inequality again gives
$$
\int_{\omega_R}\chi_{{\mathcal V}_2}
\Big|\frac{\pd Q}{\pd \zeta}(v)\Big||\nabla P_t f| 
\leq
C\bigg(
\int_{\omega_R}|P_t f|^{2(p-1)}
\bigg)^{1/2}
\nor{\nabla P_tf}_{L^2(\omega_R)}\,.
$$
Since the semigroup $P_t$ is $L^{\infty}$-contractive \cite[Proposition 4.32]{O} we can continue as
$$
\leqslant
C\nor{f}_\infty^{p-1}|\omega_R|^{1/2}
\nor{\nabla P_tf}_{L^2(\omega_R)}\,.
$$
Note that $|\omega_R|= C\, R^n$. Together with \eqref{offdiag1} this implies
$$
\leqslant
\frac{C(f,p,n)}{\sqrt t}\,R^{n/2}e^{-\frac ct \, R^2}\,,
$$
which again disappears after performing $\lim_{R\rightarrow\infty}\int_0^T\,dt$.

\item
Finally, the estimates with respect to $\eta$ can be treated without splitting the cases $x\in{\mathcal V}_{1,2}(t)$.
Indeed, for any $(x,t)\in\R^{n+1}_+$ we have, by \eqref{rupkina} and \eqref{mravalzhamieri},
$$
\aligned \Big|\frac{\pd Q}{\pd \eta}(v)\Big|\leq 
C(|P_t g|+1)
\,.
\endaligned
$$
Thus we got a sum of two terms, of which the first is  the same as in \eqref{lile}, hence it
become zero after multiplying with $|\nabla P_tg|$, integrating and taking the limit. As for the last remaining term (i.e. the constant $1$), the H\"older's inequality and \eqref{offdiag1} give
$$
\int_{\omega_R} 
|\nabla P_t g|
\leqslant
\frac{|\omega_R|^{1/2}}{\sqrt t}\,
\nor{{\sqrt t}\nabla P_t g}_{L^2(\omega_R)}
\leqslant
\frac{C(g,n)}{\sqrt t}\,R^{n/2}e^{-\frac ct \, R^2}\,.
$$

\end{enumerate}
These estimates combine into \eqref{gledala}, which finishes the proof. 
\end{proof}

\begin{proof}[Proof of Proposition \ref{planinom}.]
In view of Lemma \ref{byparts} it is enough to prove 
\begin{equation*}
\label{spasenie}
\limsup_{R\rightarrow\infty}
\int_0^T\int_{\R^n} \psi_R(x) \frac{\pd b }{\pd t}(x,t)\,dx\,dt
\leq \|f\|_p^p +
\|g\|_q^q\,.
\end{equation*}
We start by another integration by parts:
$$
\aligned
\int_0^T\int_{\R^n} \psi_R(x) \frac{\pd b }{\pd t}(x,t)\,dx\,dt
& =
\int_{\R^n}\psi_R(x)\,[b (x,T)-b (x,0)]\, dx \,.
\endaligned
$$
Remember that $Q$ was a nonpositive function, therefore $b(x,T)\leqslant 0$, by \eqref{zelen bore}.
Moreover, \eqref{zelen bore} and \eqref{rupkina} also mean we can continue as 
$$
\leqslant
 \frac12\int_{\R^n}\phi(|f(x)|,|g(x)|)\,\psi_R(x)\, dx \,.
$$
By Theorem \ref{bejaz} \eqref{anterija} we get
$$
\leqslant
\int_{\R^n}(|f(x)|^p+|g(x)|^q)\,\psi_R(x)\, dx \,,
$$
which tends to $\nor{f}_p^p+\nor{g}_q^q$ as $R\rightarrow\infty$, by the dominated convergence theorem.
\end{proof}

\end{document}